\documentclass[12pt,reqno]{article}

\usepackage[usenames]{color}
\usepackage{amssymb,amsmath,amssymb}
\usepackage{bbm}

\usepackage[colorlinks=true,
linkcolor=webgreen,
filecolor=webbrown,
citecolor=webgreen]{hyperref}

\definecolor{webgreen}{rgb}{0,.5,0}
\definecolor{webbrown}{rgb}{.6,0,0}

\usepackage{color}
\usepackage{fullpage}
\usepackage{float}

\usepackage{amsthm}
\usepackage{amsfonts}
\usepackage{latexsym}
\usepackage{epsf}

\newcommand{\seqnum}[1]{\href{http://oeis.org/#1}{\underline{#1}}}
\newcommand{\goesto}{\rightarrow} 
\newcommand{\nn}{\mathbb{N}}
\newcommand{\wo}{\backslash} 
\newcommand{\set}[1]{\{#1\}}
\newcommand{\sd}{\,|\,} 
\newcommand{\abs}[1]{\bigl\lvert #1\bigr\rvert}
\newcommand{\divides}[2]{#1\mid#2}
\newcommand{\z}{\mathbb{Z}} 

\DeclareMathOperator{\Exp}{E} 

\begin{document}

\begin{center}
\epsfxsize=4in
\end{center}

\theoremstyle{plain}
\newtheorem{theorem}{Theorem}
\newtheorem{corollary}[theorem]{Corollary}
\newtheorem{lemma}[theorem]{Lemma}
\newtheorem{proposition}[theorem]{Proposition}

\theoremstyle{definition}
\newtheorem{definition}[theorem]{Definition}
\newtheorem{example}[theorem]{Example}
\newtheorem{conjecture}[theorem]{Conjecture}

\theoremstyle{remark}
\newtheorem{remark}[theorem]{Remark}

\begin{center}
\vskip 1cm{\LARGE\bf
Some Elementary Congruences for 
\vskip .11in
the Number of Weighted Integer
  Compositions  
}
\vskip 1cm
\large
Steffen Eger \\
Computer Science Department \\
Goethe University Frankfurt am Main\\
60325 Frankfurt am Main\\
Germany \\
\href{mailto:eger.steffen@gmail.com}{\tt eger.steffen@gmail.com} \\
\ \\
\end{center}

\vskip .2 in

\begin{abstract}
An integer composition of a nonnegative integer $n$ is a tuple
$(\pi_1,\ldots,\pi_k)$ of nonnegative integers whose sum is $n$; the
$\pi_i$'s are 
called the \emph{parts} of the
composition. For fixed number $k$ of parts, the number of
\emph{$f$-weighted} 
integer compositions (also called \emph{$f$-colored integer
  compositions} in
the literature), in which each part size $s$ may occur in $f(s)$
different colors, is given by the \emph{extended binomial coefficient}
$\binom{k}{n}_{f}$. We derive several congruence properties for
$\binom{k}{n}_{f}$, most of which are analogous to those for ordinary
binomial coefficients. Among them is the parity of $\binom{k}{n}_{f}$,
Babbage's congruence, Lucas' theorem, etc. We also give congruences
for $c_{f}(n)$, the number of $f$-weighted integer compositions with
arbitrarily many parts, and for extended binomial coefficient sums.  
We close with an application of our results to prime criteria for
weighted integer compositions. 
\end{abstract}

\section{Introduction}
An \emph{integer composition} (ordered partition) of a nonnegative
integer $n$ is a tuple 
$(\pi_1,\ldots,\pi_k)$ of nonnegative integers whose sum is $n$; the
$\pi_i$'s are called the \emph{parts} of the composition. We
call an integer composition of $n$ \emph{$f$-weighted}, for a
function $f:\nn\goesto\nn$, whereby $\nn$ denotes the set of
nonnegative integers, if each {part size} $s\in\nn$ may occur in
$f(s)$ different colors in the composition. 
If $f$ is the
indicator
function of a subset $A\subseteq\nn$, this yields the
so-called \emph{$A$-restricted integer compositions}
\cite{Heubach:2004};\footnote{In particular, if $f$ is the
  indicator function of the nonnegative integers, then this yields the
so-called \emph{weak compositions} and if $f$ is the indicator
function of the positive integers, this yields the ordinary integer
compositions.} if
$f(s)=s$, this yields the so-called \emph{$s$-colored compositions}
\cite{Agarwal:2000}. 

To illustrate, let $f(1)=f(2)=f(3)=1$ and $f(9)=3$, and let
$f(s)=0$, for all $s\in\nn\wo\set{1,2,3,9}$. Then, there are $4!\cdot
3+4\cdot 3=84$ different $f$-weighted integer compositions of $n=15$ with
exactly $k=4$ parts, among them, 
\begin{align*}
  (1,3,2,9^1),(1,3,2,9^2),(1,3,2,9^3), 
\end{align*}
where we superscript the different 
colors of part size $9$. 
Obviously, $k=4$ divides $4!\cdot 3+4\cdot 3$, and
this is not coincidental and does not
depend upon $f$, as we will show. 
More generally, we derive several
divisibility properties of the number of $f$-weighted integer
compositions.  
First, 
after reviewing some introductory 
background regarding weighted integer compositions, their relations to
extended binomial coefficients, and elementary properties of weighted
integer compositions in Section \ref{sec:fixedk},  
we consider
divisibility properties for 
$f$-weighted integer compositions 
with a fixed number $k$ of parts in Section \ref{sec:fixed}. 
Then, in Section \ref{sec:arbitrary}, we combine several known
results to derive divisibility properties for the number of
$f$-weighted integer compositions of $n$ with arbitrarily many parts. 
In the same section, we also specify divisibility properties for
extended binomial coefficient sums. Lastly, in Section
\ref{sec:applications}, we close with an application of
our results to prime criteria for  
weighted integer compositions.

To place our work in some context, we note that 
there is a large body of recent results on integer compositions. To
name just a few examples, Heubach and Mansour \cite{Heubach:2004}
investigate 
generating functions for the so-called $A$-restricted compositions;
Sagan \cite{Sagan:2009} considers doubly restricted integer
compositions; 
Agarwal \cite{Agarwal:2000}, Narang and Agarwal \cite{Narang:2008},
Guo \cite{Guo:2012}, Hopkins \cite{Hopkins:2012}, 
Shapcott \cite{Shapcott:2012,Shapcott:2013}, 
and Mansour and Shattuck 
\cite{Mansour:Shattuck:2014} study results for $s$-colored integer
compositions. Mansour, Shattuck, and Wilson \cite{Mansour:Shattuck:Wilson:2014}, Munagi \cite{Munagi:2012}, 
and Munagi and Sellers \cite{Munagi:toappear} count the number 
of compositions of an integer in which (adjacent) parts satisfy 
congruence relationships. Probabilistic results for (restricted) integer
compositions are provided in Ratsaby \cite{Ratsaby:2008}, Neuschel
\cite{Neuschel:2014}, and in Banderier and Hitczenko
\cite{Banderier:2012}, among many others.  
Mihoubi \cite{Mihoubi:2009} studies congruences for the partial Bell polynomials,
which may be considered special cases of weighted integer
compositions \cite{Eger:2015}. 
Classical results on weighted
integer compositions are, for example, provided in Hoggatt and Lind
\cite{Hoggatt:1969} and some congruence relationships for classical extended
binomial coefficients are given, e.g., in Bollinger and Burchard
\cite{Bollinger:1990} and 
Bodarenko \cite{Bodarenko:1990}. 

\section{Number of $f$-weighted integer compositions
  with fixed number of parts}\label{sec:fixedk} 
For $k\ge 0$ and $n\ge 0$, consider the coefficient of $x^n$ of the
polynomial or power series 
\begin{align}\label{eq:power}
  \Bigl(\sum_{s\in\nn} f(s)x^s\Bigr)^k,
\end{align}
and denote it by $\binom{k}{n}_{f}$. Our first theorem states that $\binom{k}{n}_{f}$ denotes the combinatorial object we are investigating in this work, $f$-weighted integer compositions. 

\begin{theorem}\label{theorem:main}
  The number $\binom{k}{n}_{f}$ denotes the number of $f$-weighted integer compositions of $n$ with $k$ parts. 
\end{theorem}
\begin{proof}
  Collecting terms in \eqref{eq:power}, we see that $[x^n]g(x)$, for $g(x)= (\sum_{s\in\nn} f(s)x^s)^k$,  is given as
  \begin{align}\label{eq:comp}
    \sum_{\pi_1+\cdots+\pi_k=n}f(\pi_1)\cdots f(\pi_k),
  \end{align}
  where the sum is over all nonnegative integer solutions to
  $\pi_1+\dotsc+\pi_k=n$. This proves the theorem. 
\end{proof}
Theorem \ref{theorem:main} has appeared, for example, in Shapcott
\cite{Shapcott:2012}, Eger \cite{Eger:2013}, or, much earlier, in Hoggatt and Lind \cite{Hoggatt:1969}.   
Note that $\binom{k}{n}_{f}$, which has also been referred to as
\emph{extended binomial coefficient} in the literature 
\cite{Fahssi:2012}, 
generalizes many interesting combinatorial objects, such as the
binomial coefficients (for $f(0)=f(1)=1$ and $f(s)=0$, for $s>1$) \seqnum{A007318},
trinomial coefficients \seqnum{A027907}, etc. 

We now list four relevant properties of the $f$-weighted integer compositions,
which we will make use of in the proofs of congruence properties later
on.  Throughout this work, we will denote the ordinary binomial
coefficients, i.e., when $f(0)=f(1)=1$ and $f(s)=0$ for all $s>1$, by
the standard notation
$\binom{k}{n}$. 
\begin{theorem}[Properties of $f$-weighted integer
    compositions]\label{prop:convolution} 
  Let $k,n\ge 0$. Then, the following hold true:
  \begin{align}
    \label{eq:repr}
    \binom{k}{n}_f\: &\:= \sum_{\substack{k_0+\cdots+k_n=k\\0\cdot
        k_0+\cdots n\cdot
        k_n=n}}\binom{k}{k_0,\ldots,k_n}\prod_{i=0}^nf(i)^{k_i}\\
    \label{eq:vandermonde}
    \binom{k}{n}_{f}\: &\:=
    \sum_{\mu_1+\cdots+\mu_r=n}\binom{k_1}{\mu_1}_{f}\binom{k_2}{\mu_2}_{f}\cdots\binom{k_r}{\mu_r}_f
\\
\label{eq:absorption}
\binom{k}{n}_{f}\: &\:=
\frac{k}{in}\sum_{s\in\nn}s\binom{i}{s}_{f}\binom{k-i}{n-s}_{f}\\
\label{eq:rec}
\binom{k}{n}_{f}\: &\:=
    \sum_{i\in\nn}f(\ell)^i\binom{k}{i}\binom{k-i}{n-\ell i}_{f_{|f(\ell)=0}}
  \end{align}
  In \eqref{eq:repr}, the sum is over all solutions in nonnegative
  integers $k_0,\ldots,k_n$ of $k_0+\cdots+k_n=n$ and $0\cdot k_0+\cdots+nk_n=n$, and 
  $\binom{k}{k_0,\ldots,k_n}=\frac{k!}{k_0!\cdots k_n!}$ denote the
  multinomial coefficients. In \eqref{eq:vandermonde}, which is also
  sometimes called \emph{Vandermonde convolution} 
  \cite{Fahssi:2012}, 
  the sum is over
  all solutions in nonnegative integers $\mu_1,\ldots,\mu_r$ of
  $\mu_1+\cdots+\mu_r=n$, and the relationship holds for any fixed
  composition $(k_1,\ldots,k_r)$ of $k$, for $r\ge 1$. In
  \eqref{eq:absorption}, $i$ is an integer satisfying $0<i\le k$. In
  \eqref{eq:rec}, $\ell\in\nn$ and by $f_{|{f(\ell)=0}}$ we
  denote the function $g:\nn\goesto\nn$ for which $g(s)=f(s)$, for
  all $s\neq \ell$, and $g(\ell)=0$. 
\end{theorem}
\begin{proof}
  \eqref{eq:repr} follows from rewriting the sum in
  \eqref{eq:comp} as a summation over integer partitions rather than
  over integer compositions and then adjusting the factors in the sum.  
  \eqref{eq:vandermonde} and \eqref{eq:rec}
  have straightforward combinatorial
  interpretations, and proofs can be found in Fahssi \cite{Fahssi:2012} and Eger
  \cite{Eger:2013}. For a proof of \eqref{eq:absorption}, note first that
  $\binom{k}{n}_f$ also represents the distribution of the sum of
  i.i.d.\ nonnegative integer-valued  
  random variables $X_1,\ldots,X_k$. Namely, let
  $P[X_i=s]=\frac{f(s)}{\sum_{s'\in\nn}f(s')}$ (wlog, we may assume
  $\sum_{s'\in\nn}f(s')$ to be finite). Then, using \eqref{eq:comp},
  \begin{align*}
    P[X_1+\cdots+X_k=n] = 
    \sum_{\pi_1+\cdots+\pi_k=n}P[X_1=\pi_1]\cdots P[X_k=\pi_k] = \left(\frac{1}{\sum_{s'\in\nn}f(s')}\right)^k\binom{k}{n}_{f}.
  \end{align*}
  Thus, it suffices to prove
  \eqref{eq:absorption} for sums of random variables. For $0<i\le k$, let
  $S_i$ denote the partial sum $X_1+\cdots+X_i$. Then, consider the
  conditional expectation $\Exp[S_i\sd S_k=n]$, for which the relation
  \begin{align*}
    \Exp[S_i\sd S_k=n] = \frac{n}{k}i,
  \end{align*}
  holds, by independent and identical distribution of
  $X_1,\ldots,X_k$. Moreover, by definition of conditional
  expectation, we have that 
  \begin{align*}
  \Exp[S_i\sd S_k=n] = \sum_{s\in\nn} s \frac{P[S_i=s,S_k=n]}{P[S_k=n]} = \sum_{s\in\nn} s\frac{P[S_i=s]\cdot P[S_{k-i}=n-s]}{P[S_k=n]}.
\end{align*}
Combining the two identities for $\Exp[S_i\sd S_k=n]$ and rearranging
yields \eqref{eq:absorption}.
\end{proof}
\begin{remark}\label{rem:triangle}
Note the following important special case of 
\eqref{eq:vandermonde} which results when we let $r=2$ and $k_1=1$ and
$k_2=k-1$,
\begin{align*}
  \binom{k}{n}_{f}=\sum_{\mu=0}^{n}f(\mu)\binom{k-1}{n-\mu}_f,
\end{align*}
which establishes that the quantities 
$\binom{k}{n}_{f}$ 
may be perceived of as generating a Pascal-triangle-like array in
which entries 
in row $k$ are weighted sums of the entries in row
$k-1$. 
To 
illustrate, the left-justified triangle for $f(0)=5$, $f(1)=0$,
$f(2)=2$, $f(3)=1$, $f(x)=0$, for all $x>3$, starts as
\begin{table}[!htb]
\begin{center}
\begin{tabular}{c|ccccccccccc}
  $k\backslash n$
  & 0 & 1 & 2 & 3 & 4 & 5 & 6 & 7 & 8 & 9 &
  $\cdots$\\ \hline 
  0 & 1 \\
  1 & 5 & 0 & 2 & 1\\
  2 & 25 & 0 & 20 & 10 & 4 & 4 & 1\\
  3 & 125 & 0 & 150 & 75 & 60 & 60 & 23 & 12 & 6 & 1\\
  $\vdots$& $\vdots$ & $\cdots$ & $\cdots$& $\cdots$& $\cdots$&
  $\cdots$& $\cdots$ & $\cdots$ & $\cdots$& $\cdots$&
  $\ddots$ \\ 
\end{tabular}
\end{center}
\end{table}
\end{remark}
We also note the following special cases of $\binom{k}{n}_f$, which we will make use of in
Section \ref{sec:fixed}.
\begin{lemma}\label{lemma:f0}
  For all $x,k\in\nn$, we have that  
  \begin{align*}
    \binom{k}{0}_f &= f(0)^k,\\
    \binom{k}{1}_f &= kf(1)f(0)^{k-1},\\
    \binom{1}{x}_f &= f(x),\\ 
    \binom{0}{x}_f &= \begin{cases}1, &
      \text{if } x=0;\\ 0, & \text{otherwise}.\end{cases}
  \end{align*}
\end{lemma}

\section{Some elementary divisibility properties of the number of
  $f$-weighted 
  integer compositions with fixed number of parts}\label{sec:fixed}

\begin{theorem}[Parity of extended binomial
    coefficients]\label{theorem:parity} 
  \begin{align*}
    \binom{k}{n}_{f} \equiv
    \begin{cases}
      0 \pmod{2}, & \text{if $k$ is even and $n$ is odd};\\
      \binom{k/2}{n/2}_{f} \pmod{2}, & \text{if $k$ is even and
        $n$ is even};\\
      \sum_{s\ge 0}f(2s+p(n))\binom{\lfloor k/2\rfloor}{\lfloor
        n/2\rfloor-s}_{f}\pmod{2}, & \text{if $k$ is odd};
    \end{cases}
  \end{align*}
  where we let $p(n)=0$ if $n$ is even and $p(n)=1$ otherwise. 
\end{theorem}
\begin{proof} 
  We distinguish three cases. 
  \begin{itemize}
    \item Case $1$: 
	Let $k$ be even and $n$ odd. In 
        \eqref{eq:absorption} in Theorem \ref{theorem:main}
        with $i=1$, 
        multiply both sides by $n$. 
        If $k$ is even, the right-hand
      side is even, and thus, if $n$ is odd, $\binom{k}{n}_{f}$ must
      be even.  
    \item Case $2$: Let $k$ be even and $n$ even. Consider the
      Vandermonde convolution in the case when $r=2$ and $j=k/2$. Then,
      \begin{align*}
        \binom{k}{n}_{f} &=
        \sum_{\mu+\nu=n}\binom{k/2}{\mu}_{f}\binom{k/2}{\nu}_{f}  = 2\sum_{0\le
        \mu<n/2}
        \binom{k/2}{\mu}_{f}\binom{k/2}{n-\mu}_{f}+\binom{k/2}{n/2}_{f}
        \\ &\equiv \binom{k/2}{n/2}_{f}\pmod{2}. 
      \end{align*}
    \item Case $3$: Let $k$ be odd. Then $k-1$ is even. Thus,
      the Vandermonde convolution with $j=1$, $r=2$ implies 
      \begin{align*}
        \binom{k}{n}_{f}=\sum_{s\in
          \nn}f(s)\binom{k-1}{n-s}_{f}\equiv \sum_{\set{s\in \nn\sd 
        n-s \text{ is even}}}
        f(s)\binom{\frac{k-1}{2}}{\frac{n-s}{2}}_{f}
        \pmod{2},
      \end{align*}
      where we use Case $1$ and Case $2$ in the last
      congruence. Hence, if $n$ is even, 
      \begin{align*}
        \binom{k}{n}_{f}\equiv \sum_{s\ge 0}
        f(2s)\binom{\lfloor k/2\rfloor}{\frac{n}{2}-s}_{f}\pmod{2}, 
      \end{align*}
      and if $n$ is odd,
      \begin{align*}
        \binom{k}{n}_{f}\equiv \sum_{s\ge 0}
        f(2s+1)\binom{\lfloor k/2\rfloor}{\lfloor n/2
          \rfloor-s}_{f}\pmod{2}. 
      \end{align*}
  \end{itemize}
\end{proof}
\begin{example}
  Let $f(0)=3,f(1)=2,f(2)=1$ and $f(s)=0$ for all $s>2$. Then, by
  Theorem \ref{theorem:parity},
  \begin{align*}
    \binom{13}{14}_f &\equiv
    f(0)\binom{6}{7}_f+f(2)\binom{6}{6}_f+\underbrace{f(4)\binom{6}{5}_f+\cdots}_{=0}
    \equiv 0+\binom{3}{3}_f\\
    &\equiv
    f(1)\binom{1}{1}_f+\underbrace{f(3)\binom{1}{0}_f+\cdots}_{=0} =
    2\binom{1}{1}_f \equiv 0\pmod{2},
  \end{align*}
  and, in fact, $\binom{13}{14}_f=289,159,780$. 
\end{example}

\begin{theorem}\label{theorem:prime1}
  Let $p$ be prime. Then 
  \begin{align*}
    \binom{p}{n}_{f} \equiv 
    \begin{cases}
      f(r) \pmod{p}, & \text{if $n=pr$ for some $r$};\\ 
      0 \pmod{p}, & \text{else.}
    \end{cases}
  \end{align*}
\end{theorem}
We sketch three proofs of Theorem \ref{theorem:prime1}, 
a combinatorial proof and two proof sketches based on identities in
Theorem \ref{prop:convolution}. The first proof is based on the
following lemma \cite{Anderson:2005}. 
\begin{lemma}\label{lemma:comb}
  Let $S$ be a finite set, let $p$ be prime, and suppose $g:S\goesto
  S$ has the property that $g^p(x)=x$ for any $x$ in $S$, where $g^p$
  is the $p$-fold composition of $g$. Then
  $\abs{S}\equiv\abs{F}\pmod{p}$, where $F$ is the set of fixed points
  of $g$. 
\end{lemma}
\begin{proof}[Proof of Theorem \ref{theorem:prime1}, 1]
  For an $f$-weighted integer composition of $n$ with $p$ parts, let $g$
  be the operation that shifts all parts one to the right, modulo
  $p$. In other words, $g$ maps (denoting different colors by superscripts)
  $(\pi_1^{\alpha_1},\pi_2^{\alpha_2},\ldots,\pi_{p-1}^{\alpha_{p-1}},\pi_p^{\alpha_p})$
  to 
  \begin{align*}
    (\pi_p^{\alpha_p},\pi_1^{\alpha_1},\pi_2^{\alpha_2},\ldots,\pi_{p-1}^{\alpha_{p-1}}).
  \end{align*}
  Of course, applying $g$ $p$ times yields the original $f$-colored
  integer composition, that is, $g^{p}(x)=x$ for all $x$. 
	We may thus apply Lemma \ref{lemma:comb}. If $n$ allows a
  representation $n=pr$ for some suitable $r$, $g$ has exactly $f(r)$
  fixed points, namely, all compositions $\underbrace{(r^{1},\ldots,
    r^{1})}_{p \text{ times}}$ to
  $\underbrace{(r^{f(r)},\ldots,r^{f(r)})}_{p \text{ times}}$. Otherwise,
  if $n$ has no such 
  representation, $g$ has no fixed points. 
  This proves the theorem. 
\end{proof}
\begin{proof}[Proof of Theorem \ref{theorem:prime1}, 2]
  We apply \eqref{eq:rec} in Theorem 
  \ref{prop:convolution}. Since for the ordinary binomial
  coefficients, the relation $\binom{p}{n}\equiv
  0\pmod{p}$ holds for all $1\le n\le p-1$ and
  $\binom{p}{0}=\binom{p}{p}=1$, we have 
  \begin{align*}
    \binom{p}{n}_{f}\equiv
    \binom{p}{n}_{f_{|f(\ell)=0}}+f(\ell)^p\binom{0}{n-\ell
      p}_{f_{|f(\ell)=0}}\equiv \binom{p}{n}_{f_{|f(\ell)=0}}+f(\ell)\binom{0}{n-\ell
      p}_{f_{|f(\ell)=0}} \pmod{p}, 
  \end{align*}
  for any $\ell$ and where the last congruence is due to Fermat's
  little theorem. Therefore, if $n=rp$ for some $r$, then
  $\binom{p}{n}_f\equiv  \binom{p}{n}_{f_{|f(r)=0}}+f(r)\pmod{p}$ and
  otherwise $\binom{p}{n}_f \equiv
  \binom{p}{n}_{f_{|f(\ell)=0}}\pmod{p}$ for any $\ell$. 
  Now, the theorem follows inductively. 
\end{proof}
\begin{proof}[Proof of Theorem \ref{theorem:prime1}, 3] 
  Finally, we can use  
  \eqref{eq:repr} in Theorem \ref{prop:convolution} in conjunction
  with the following property of multinomial
  coefficients (see, e.g., Ricci \cite{Ricci:1931}), namely, 
  \begin{align}\label{eq:multi}
    \binom{k}{k_0,\ldots,k_n}\equiv
    0\pmod{\frac{k}{\gcd{(k_0,\ldots,k_n)}}}. 
  \end{align}
  From this, whenever $n\neq pr$, $\binom{p}{n}_f\equiv 0\pmod{p}$
  since for all terms in the summation in \eqref{eq:repr},
  $\gcd{(k_0,\ldots,k_n)}=1$. Otherwise, if $n=pr$ for some $r$, then
  $\gcd{(k_0,\ldots,k_n)}>1$ precisely when one of the $k_i$'s is $p$
  and the remaining are zero. Since also $0k_0+\cdots+nk_n=n=rp$, this
  can only occur when $k_r=p$. Hence, $\binom{p}{rp}_f \equiv
  \binom{p}{0,\ldots,p,\ldots,0}f(r)^p\equiv f(r)\pmod{p}$. 
\end{proof}

The next immediate corollary generalizes the congruence $(1+x)^p\equiv 1+x^p\pmod{p}$, for $p$ prime. 
\begin{corollary}
  Let $p$ be prime. Then,  
  \begin{align*}
      \left(\sum_{s\in\nn}f(s)x^s\right)^p = \sum_{n\in\nn}\binom{p}{n}_{f}x^n\equiv \sum_{r\in\nn}f(r)x^{pr}\pmod{p}.
  \end{align*}
\end{corollary}
\begin{corollary}\label{cor:}
  Let $k,s\ge 0$ and $p$ prime. Then, 
  \begin{align*}
    \binom{k+sp}{j}_f \equiv \binom{k}{j}_ff(0)^{sp}\pmod{p},
  \end{align*}
  for $0\le j<p$. 
\end{corollary}
\begin{proof}
  By the Vandermonde convolution, \eqref{eq:vandermonde}, we have 
  \begin{align*}
    \binom{k+sp}{j}_f = \sum_{x+y=j} \binom{k}{x}_f\binom{sp}{y}_f.
  \end{align*}
  Now, again by the Vandermonde convolution, $\binom{sp}{y}_f =
  \sum_{x_1+\cdots+x_s=y}\binom{p}{x_1}_f\cdots\binom{p}{x_s}_f$. Since
  $0\le y\le j< p$, 
  the product $\prod\binom{p}{x_i}_f$ is divisible by $p$
  by Theorem \ref{theorem:prime1} 
  whenever 
  $x_1=\cdots=x_s=0$ does not hold. Therefore,
  \begin{align*}
    \binom{k+sp}{j}_f \equiv \binom{k}{j}_f\binom{sp}{0}_f =
    \binom{k}{j}_ff(0)^{sp}\pmod{p}, 
  \end{align*}
  by Lemma \ref{lemma:f0}.
\end{proof}
\begin{corollary}\label{cor:pplus1}
  Let $p$ be prime and $0\le m,r$ with $r<p$. Then,
  \begin{align*}
    \binom{p+1}{mp+r}_f \equiv \sum_{s\ge 0}f(r+sp)f(m-s)\pmod{p}. 
  \end{align*}
\end{corollary}
\begin{proof}
  This follows from $\binom{p+1}{n}_f = \sum_{s\ge
    0}f(s)\binom{p}{n-s}_f$ and Theorem \ref{theorem:prime1}. 
\end{proof}
\begin{remark}
  Similar results as in Corollary \ref{cor:pplus1} can be derived for
  $\binom{p+2}{mp+r}_f$, etc., but the formulas become more
  complicated. 
\end{remark}
With similar arguments as before, we can also prove a stronger version
of Theorem \ref{theorem:prime1}, namely:
\begin{theorem}\label{theorem:prime2}
  Let $p$ be prime and let $m\ge 1$. Then
  \begin{align*}
    \binom{p^m}{n}_{f} \equiv 
    \begin{cases}
      f(r) \pmod{p}, & \text{if $n=p^mr$ for some $r$};\\ 
      0 \pmod{p}, & \text{else.}
    \end{cases}
  \end{align*}
\end{theorem}

We call the next congruence Babbage's congruence, since Charles
Babbage was apparently the first to assert the respective congruence
in the case of ordinary binomial coefficients \cite{Babbage:1819}. 
\begin{theorem}[Babbage's congruence]\label{theorem:babbage}
 Let $p$ be prime, and let $n$ and $m$ be nonnegative integers. Then
\begin{align*}
  \binom{np}{mp}_{f} \equiv \binom{n}{m}_{g}\pmod{p^2},
\end{align*}  
whereby $g$ is defined as $g(r)=\binom{p}{rp}_{f}$, for all $r\in\nn$. 
\end{theorem}
\begin{proof}
By the Vandermonde convolution, we have
\begin{align}\label{eq:p2}
  \binom{np}{mp}_f = \sum_{k_1+\cdots+k_n=mp}\binom{p}{k_1}_f\cdots\binom{p}{k_n}_f
\end{align}
Now, by Theorem \ref{theorem:prime1}, $p$ divides
$\binom{p}{x}_{f}$ 
whenever $x$ is not of the form $x=pr$. Hence, modulo
$p^2$, the only terms that contribute to the sum are those for which
at least $n-1$ $k_i$'s are of the form $k_i=r_ip$. Since the
$k_i$'s must sum to $mp$, this implies that all $k_i$'s are of the
form $k_i=r_ip$, for $i=1,\ldots,n$. Hence, 
modulo $p^2$, \eqref{eq:p2} becomes 
\begin{align*}
   \sum_{{r_1+\cdots+r_n=m}}\prod_{i=1}^n \binom{p}{r_ip}_{f} =\sum_{{r_1+\cdots+r_n=m}}\prod_{i=1}^n g(r_i),
\end{align*}
The last sum is precisely $\binom{n}{m}_{g}$.
\end{proof}
\begin{corollary}\label{cor:babbage}
  Let $r\ge 0$ and let $p$ be prime. Then 
  \begin{align*}
  \binom{pr}{p}_f \equiv \binom{p}{p}_ff(0)^{p(r-1)}r\pmod{p^2}.
  \end{align*}
\end{corollary}
\begin{proof}
  This follows by combining Theorem \ref{theorem:babbage} and Lemma
  \ref{lemma:f0}. 
\end{proof}
Now, we consider the case when $x$ in $\binom{np}{x}_f$ is not of the
form $mp$ for some $m$. 
\begin{theorem}\label{theorem:babbage2}
  Let $p$ be prime and let $s,r$ be nonnegative integers. 
  Let $p$ not divide $r$. Then, 
  \begin{align*}
    \binom{sp}{r}_f \equiv
    s\cdot
    \sum_{\set{0\le i_1\le r\sd r-i_1=m_{i_1}p}}
    \binom{p}{i_1}_f\binom{s-1}{m_{i_1}}_g \pmod{p^2},
  \end{align*}
  where $g$ is as defined in Theorem \ref{theorem:babbage}. 
\end{theorem}
\begin{proof}
  By the Vandermonde convolution, \eqref{eq:vandermonde}, we find
  that
  \begin{align*}
    \binom{sp}{r}_f =
    \sum_{i_1+\cdots+i_s=r}\binom{p}{i_1}_f\cdots\binom{p}{i_s}_f = \sum_{i_1=0}^r \binom{p}{i_1}_f
    \sum_{i_2+\cdots+i_s=r-i_1}\binom{p}{i_2}_f\cdots\binom{p}{i_s}_f. 
  \end{align*}
  Now, $\binom{p}{x}_f\equiv 0\pmod{p}$ whenever $x$ is not of the form
  $x=ap$, by Theorem \ref{theorem:prime1}. Thus, modulo $p^2$, the
  above RHS 
  is $\equiv 0$ 
  unless for at least $s-1$ factors $\binom{p}{i_j}_f$ we have that
  $i_j=a_jp$ for some $a_j$. Not all $s$ factors can be of the form
  $a_jp$, since otherwise $i_1+\cdots+i_s=p(a_1+\cdots+a_s)=r$,
  contradicting that $p\nmid r$. Hence, exactly $s-1$ factors must be
  of the form $a_jp$, and therefore, 
  \begin{align*}
    \binom{sp}{r}_f & \equiv s\sum_{i_1=0,p\nmid i_1}^r \binom{p}{i_1}_f
    \sum_{a_2p+\cdots+a_sp=r-i_1}\binom{p}{a_2p}_f\cdots\binom{p}{a_sp}_f
    \\
    &= s\sum_{i_1=0,p\nmid i_1}^r \binom{p}{i_1}_f
    \sum_{a_2p+\cdots+a_sp=r-i_1}g(a_2)\cdots g(a_s)\pmod{p^2},
  \end{align*}
  Now, the
  equation $p(a_2+\cdots+a_s)=a_2p+\cdots+a_sp=r-i_1$ has solutions
  only when $\divides{p}{r-i_1}$, that is, when there exists $m_{i_1}$
  such that $r-i_1=m_{i_1}p$. 
\end{proof}
\begin{corollary}\label{cor:1}
  Let $p$ be prime, $s\ge 0$ and let $0\le r\le p$. Then,
  \begin{align*}
    \binom{sp}{r}_f \equiv s\binom{p}{r}_f\cdot f(0)^{p(s-1)}\pmod{p^2}.
  \end{align*}
\end{corollary}
\begin{proof}
  For $r=p$, 
  this is Corollary \ref{cor:babbage}. 
  For $0\le r<p$, the proof follows from
  Theorem \ref{theorem:babbage2} by noting that $i_1=r$ and $m_{i_1}=0$ is
  the only solution to the sum constraint.
\end{proof}
Corollary \ref{cor:1} immediately implies the following: 
\begin{corollary}
  Let $0\le r,s\le p$. Then,
  \begin{align*}
    f(0)^{p(s-1)}s\binom{rp}{r}_f \equiv f(0)^{p(r-1)}r\binom{sp}{r}_f
    \pmod{p^2}. 
  \end{align*}
\end{corollary}

\begin{theorem}\label{theorem:divis}
Let $m,k,n\ge 0$ be nonnegative integers. Then
\begin{align*}
\binom{mk}{n}_{f}\equiv 0\pmod{\frac{k}{\gcd{(k,n)}}}.
\end{align*}
\end{theorem}
\begin{proof}
From \eqref{eq:absorption}, with $i=1$, write 
\begin{align*}
  \frac{1}{d}n\binom{mk}{n}_{f} = \frac{1}{d}mk\sum_{s\in\nn}sf(s)\binom{mk-1}{n-s}_{f} = \frac{k}{d}A,
\end{align*}
where $A\in\nn$, $d=\gcd(k,n)$ and note that $\gcd(k/d,n/d)=1$.
\end{proof}

\begin{theorem}\label{theorem:ms}
  Let $p$ be prime and $r\ge 1$ arbitrary. Then, 
\begin{align*}
 \binom{pr}{p}_{f}\equiv f(0)^{p(r-1)}f(1)^{p}\binom{pr}{p} \pmod{pr}. 
\end{align*}
\end{theorem}
\begin{proof}
From \eqref{eq:repr},
$\binom{pr}{p}_{f}$ can be written as 
\begin{align}\label{eq:help}
  \binom{pr}{p}_{f} =
  \sum_{\substack{k_0+\cdots+k_p=pr,\\ 0\cdot k_0+\cdots+p\cdot
      k_p=p}}\binom{pr}{k_0,\ldots,k_p}\prod_{s=0}^{p} f(s)^{k_s}. 
\end{align}
For a term in the sum, 
either $d=\gcd(k_0,\ldots,k_p)=1$ or $d=p$, since otherwise, if $1<d<p$, 
then $d\cdot (0\cdot k_0/d+\cdots p\cdot k_p/d)=p$, whence $p$ is composite, a contradiction.
Those terms on the RHS of \eqref{eq:help} for which $d=1$ contribute
nothing to the sum modulo $pr$, by \eqref{eq:multi}, so they
can be ignored.  But, from the equation $0\cdot k_0+1\cdot k_1+\cdots
p\cdot k_p=p$, the case $d=p$ precisely happens when $k_1=p$,
$k_2=\cdots=k_p=0$ and when $k_0=p(r-1)$ (from the equation
$k_0+\cdots+k_p=pr$),
whence, as
required, $\binom{pr}{p}_{f}\equiv
f(0)^{p(r-1)}f(1)^p\binom{pr}{p}\pmod{pr}$. 
\end{proof}

Recall that the ordinary binomial coefficients satisfy Lucas'
theorem, namely,
\begin{align*}
  \binom{k}{n} \equiv \prod\binom{k_i}{n_i}\pmod{p},
\end{align*}
whenever $k=\sum n_ip^i$ and $n=\sum k_ip^i$ with $0\le
n_i,k_i<p$. An analogous result has been established in Bollinger and
Burchard \cite{Bollinger:1990} for 
the classical 
{extended binomial coefficients}, 
the coefficients 
of $\left(1+x+\ldots+x^m\right)^k$. We
straightforwardly extend their result for our more general situation
of arbitrarily weighted integer compositions (general extended
binomial coefficients). 
\begin{theorem}[Lucas' theorem]
  Let $p$ be a prime and let $n=\sum_{i=0}^t n_ip^i$ and $k=\sum_{j=0}^r k_jp^j$, where $0\le n_i,k_j<p$. Then
  \begin{align*}
   \binom{k}{n}_{f} \equiv \sum_{(s_0,\ldots,s_r)}\prod_{i=0}^r \binom{k_i}{s_i}_{f}\pmod{p},
  \end{align*}
  whereby the sum is over all $(s_0,\ldots,s_r)$ that satisfy
  $s_0+s_1p+\cdots+s_rp^r=n$. 
\end{theorem}
\begin{proof}
  \begin{align*}
    \sum_{n\ge 0}\binom{k}{n}_fx^n &= \left(\sum_{s\ge
      0}f(s)x^s\right)^k = \prod_{j=0}^r\left(\sum_{s\ge
      0}f(s)x^s\right)^{k_jp^j}
    \equiv \prod_{j=0}^r \left(\sum_{s\ge 0}f(s)x^{p^js}\right)^{k_j}
    \\&= \prod_{j=0}^r\left(\sum_{m\ge
      0}\binom{k_j}{m}_fx^{p^jm}\right) = \sum_{n\ge
      0}\left(\sum_{(s_0,\ldots,s_r)}\binom{k_0}{s_0}_f\cdots\binom{k_r}{s_r}_f\right)x^n\pmod{p},
  \end{align*}
  where the third equality follows from Theorem \ref{theorem:prime2},
  and the theorem follows by comparing the coefficients of $x^n$. 
\end{proof}

Finally, we conclude this section with a theorem given in Granville
\cite{Granville:1997} which 
allows a `fast computation' of $\binom{k}{n}_f$ modulo a prime.
\begin{theorem}
Let $p$ be a prime. Then, 
\begin{align*}
\binom{k}{n}_{f} \equiv \sum_{m\ge 0}\binom{\lfloor k/p\rfloor}{\lfloor n/p\rfloor -m}_{f}\binom{k_0}{n_0+mp}_f\pmod{p},
\end{align*}
whereby $n_0$ and $k_0$ are the remainders when dividing $n$ and $k$ by $p$. 
\end{theorem}
\begin{proof}
  We have
  \begin{align*}
    \left(\sum_{s\ge 0}f(s)x^s\right)^p \equiv \sum_{s\ge 0}
    f(s)x^{ps} \pmod{p}
  \end{align*}
  by Theorem \ref{theorem:prime1} and therefore, with $k=k_0+k_1p$,
  for $0\le k_0,k_1<p$, 
  \begin{align*}
    \left(\sum_{s\ge 0}f(s)x^s\right)^{k_0+k_1p} &\equiv
    \left(\sum_{s\ge 0}f(s)x^s\right)^{k_0} \left(\sum_{s\ge 0}
    f(s)x^{ps}\right)^{\lfloor k/p \rfloor} \\
    & =
    \sum_{r,t\ge 0}\binom{\lfloor
      k/p\rfloor}{t}_f\binom{k_0}{r}_fx^{pt+r}
    \pmod{p}.
  \end{align*}
  Now, since $\binom{k}{n}_f$ is the coefficient of $x^n$ of
  $\left(\sum_{s\ge 0}f(s)x^s\right)^{k_0+k_1p}$, 
  \begin{align*}
    \binom{k}{n}_f \equiv \sum_{pt+r=n}\binom{\lfloor
      k/p\rfloor}{t}_f\binom{k_0}{r}_f\pmod{p},
  \end{align*}
  and the theorem follows after re-indexing the summation on the RHS. 
\end{proof}

\section{Divisibility of the number of $f$-weighted integer compositions of $n$ with
  arbitrary number of parts $k$, and where $n\in A$}\label{sec:arbitrary}
Here, we (briefly) consider divisibility properties for the number
$c_f(n)$ of integer compositions with 
arbitrary number of parts, i.e., $c_f(n)=\sum_{k\ge 0}
\binom{k}{n}_f$, and, in Theorems \ref{theorem:glashier} and
\ref{theorem:rowsum}, 
particular divisibility properties 
for 
the total number of all $f$-weighted integer compositions of $n\in A$,
for sets $A$, with fixed
number $k$ of parts, i.e., $\sum_{n\in A}\binom{k}{n}_f$. 

First, it is easy to establish that $c_f(n)$ is a
`generalized Fibonacci sequence', 
satisfying a weighted linear
recurrence where the weights are given by $f$. 
\begin{theorem}\label{theorem:recurrence}
  For $n\ge 1$ we have that
  \begin{align*}
    c_f(n) = \sum_{m\in\nn} f(m)c_f(n-m),
  \end{align*}
  where we define $c_f(0)=1$ and $c_f(n)=0$ if $n<0$.  
\end{theorem}
\begin{proof}
  An $f$-weighted integer composition of $n$ may end, in its last
  part, with one of the values $m=0,1,2,\ldots,n$, and $m$ may be
  colored in $f(m)$ different colors.
\end{proof}
\begin{remark}
  Of course, when $f(0)>0$, then $c_f(n)>0\implies c_f(n)=\infty$ for
  all positive 
  $n$. Hence, in the remainder, we assume that $f(0)=0$. 
\end{remark}
In special cases, e.g., when $f$ is the indicator function of
particular sets 
$B\subseteq \nn$, that is, $f(s)=\mathbbm{1}_{B}(s)=\begin{cases}1, & \text{if }
s\in B;\\ 0, & \text{otherwise}\end{cases}$, it is well-known that
$c_f(n)$ is 
closely related to the ordinary Fibonacci numbers $F_n$. For example
(see, e.g., Shapcott \cite{Shapcott:2013}):  
\begin{align*}
  c_f(n) &= F_{n+1},\quad\text{for } f=\mathbbm{1}_{\set{1,2}},\\
  c_f(n) &= F_{n-1},\quad\text{for } f=\mathbbm{1}_{\nn\wo\set{0,1}},\\
  c_f(n) &= F_n,\quad\text{for } f=\mathbbm{1}_{\set{n\in\nn \sd n
    \text{ is odd}}},\\
  c_f(n) &= F_{2n},\quad\text{for } f(s)=s=\text{Id}(s).
\end{align*}
Accordingly, it immediately follows that $c_f(n)$, in these cases,
satisfies the corresponding divisibility properties of the Fibonacci
numbers, such as the following well-known properties. 
\begin{theorem}\label{theorem:fib}
  Let $p$ be prime. Then
  \begin{align*}
    c_{\mathbbm{1}_{\set{1,2}}}(p-1)\equiv 
    c_{\mathbbm{1}_{\nn\wo\set{0,1}}}(p+1)
    \equiv c_{\mathbbm{1}_{\{n\in\nn \sd n
    \text{ is odd}\}}}(p) 
    \equiv \begin{cases}
        0,&\text{if } p=5;\\
        1,&\text{if } p\equiv\pm 1\pmod{5};\\
        -1,&\text{if } p\equiv\pm 2\pmod{5}.
      \end{cases}
    \pmod{p}.
  \end{align*}
  Moreover,
  \begin{align*}
    &
    \gcd{\bigl(c_{\mathbbm{1}_{\set{1,2}}}(m),c_{\mathbbm{1}_{\set{1,2}}}(n)\bigr)}=c_{\mathbbm{1}_{\set{1,2}}}(\gcd(m+1,n+1)-1), 
    \\
    &\gcd{\bigl(c_{\mathbbm{1}_{\nn\wo\set{0,1}}}(m),c_{\mathbbm{1}_{\nn\wo\set{0,1}}}(n)\bigr)} 
    = c_{\mathbbm{1}_{\nn\wo\set{0,1}}}(\gcd{(m-1,n-1)}+1),\\
    &\gcd\bigl(c_{\mathbbm{1}_{\{n\in\nn \sd n
    \text{ is odd}\}}}(m),c_{\mathbbm{1}_{\{n\in\nn \sd n
    \text{ is odd}\}}}(n)\bigr)  = c_{\mathbbm{1}_{\{n\in\nn \sd n
    \text{ is odd}\}}}(\gcd(m,n)),\\
   &\gcd\bigl(c_{\text{Id}}(m),c_{\text{Id}}(n)\bigr)=c_{\text{Id}}(\gcd{(m,n)}). 
  \end{align*}
\end{theorem}
\begin{remark}
  Note how Theorem \ref{theorem:fib} implies several
  interesting properties, such as $3\mid c_{\text{Id}}(4m)$ (since
  $\gcd(4m,2)=2$ and $c_{\text{Id}}(2)=3$, as $2=1+1=2^1=2^2$) or,
  similarly, $7\mid c_{\text{Id}}(4m)$, which 
  otherwise also follow from well-known congruence relationships for
  Fibonacci numbers.  
\end{remark}

When $f$ is arbitrary but zero almost everywhere ($f(x)=0$ for all
$x>m$, 
for some $m\in\nn$), then by Theorem 
\ref{theorem:recurrence}, $c_f(n)$ satisfies an $m$-th order
linear recurrence, given by
\begin{align*}
  c_f(n+m) = f(1)c_f(n+m-1)+\cdots+f(m)c_f(n). 
\end{align*}
For such sequences, Somer \cite[Theorem 4]{Somer:1987}, for instance,
states a 
congruence 
relationship which we can 
immediately apply to our situation, leading to:
\begin{theorem}\label{theorem:recgeneral}
  Let $p$ be a prime and let $b$ a nonnegative integer. Let
  $f:\nn\goesto\nn$ be zero almost everywhere, i.e., $f(x)=0$ for all
  $x>m$. Then
  \begin{align*}
    c_f(n+mp^b)\equiv
    f(1)c_f(n+(m-1)p^b)+f(2)c_f(n+(m-2)p^b)+\cdots+f(m)c_f(n)\pmod{p}. 
  \end{align*}
\end{theorem}
\begin{example}
  Let $f(1)=1$, $f(2)=3$, $f(3)=0$, $f(4)=2$. 
  Let $p=5$ and $x=20=n+mp=0+4\cdot 5$. Then,  
  \begin{align*}
    f(1)c_{f}(15)+f(2)c_{f}(10)+f(3)c_{f}(5)+f(4)c_{f}(0) &=
    290375+3\cdot 3693+0\cdot 44+2\cdot 1\\ &\equiv
    11081\equiv 1\pmod{5},
  \end{align*}
  and, indeed, $c_f(20)=22,985,976\equiv 1\pmod{5}$.
\end{example}
\begin{example}
  When $f$ `avoids' a fixed arithmetic sequence, i.e., $f(s)=1$
  whenever $s\notin\set{a+mj\sd j\in\nn}$, for $a,m\in\nn$ fixed, and
  otherwise $f(s)=0$, then
  $c_f(n)$ likewise satisfies a linear recurrence \cite{Robbins:toappear}, namely,
  \begin{align*}
    c_f(n+m) &=
    c_f(n+m-1)+\cdots+c_f(n+m-a+1)+c_f(n+m-a-1)+\cdots
    \\ &+c_f(n+1)+2c_f(n), 
  \end{align*}
  and so Theorem \ref{theorem:recgeneral} applies likewise. 
\end{example}

Finally, we consider the number of $f$-weighted compositions, with
fixed number of parts, of \emph{all} numbers $n$ in some particular
sets $A$. Introduce the following notation:
\begin{align*}
  {k\brack r}_{m,f} = \sum_{\overset{n\ge 0}{n\equiv r\pmod{m}}}\binom{k}{n}_{f}.
\end{align*}
Note that ${k\brack r}_{m,f}$ generalizes the usual binomial sum
notation (cf.\ Sun \cite{Sun:2007}). In our context, ${k\brack r}_{m,f}$ denotes
the number of  
compositions, with $k$ parts, of $n\in A=\set{y\sd y\equiv
  r\pmod{m}}$. We note that, by the Vandermonde convolution,
${k\brack r}_{m,f}$ satisfies  
\begin{align}\label{eq:vand_sum}
  {k \brack r}_{m,f} = \sum_{s\ge 0} f(s){k-1\brack r-s}_{m,f}.
\end{align}

Our first theorem in this context goes back to J.\ W.\ L.\ Glaisher, 
and 
its proof is 
inspired by the corresponding proof for binomial sums due to Sun
(cf.\ Sun \cite{Sun:2007}, and references therein). 
\begin{theorem}[Glaisher]\label{theorem:glashier}
For any prime $p\equiv 1\pmod{m}$ and any $k\ge 1$, 
\begin{align*}
{k+p-1\brack r}_{m,f}\equiv {k\brack r}_{m,f}\pmod{p}.
\end{align*}
\end{theorem}
\begin{proof}
For $k=1$, 
\begin{align*}
{p \brack r}_{m,f} &= \sum_{n\ge 0, n\equiv r\pmod{m}}
\binom{p}{n}_f\equiv \sum_{q\ge 0,n=pq,n\equiv q\equiv
  r\pmod{m}}\binom{p}{pq}_f
\\ &\equiv \sum_{q\ge 0,q\equiv r\pmod{m}}f(q)\pmod{p},
\end{align*}
by Theorem \ref{theorem:prime1}, and, moreover, ${1\brack r}_{m,f}=\sum_{y\ge 0,y\equiv r\pmod{m}}f(y)\pmod{p}$ by definition. For $k>1$, the result follows by induction using \eqref{eq:vand_sum}. 
\end{proof}

\begin{theorem}\label{theorem:rowsum}
Let $f(s)=0$ for almost all $s\in\nn$. 
Consider ${k\brack 0}_{1,f}$, the row sum in row $k$, or,
equivalently, the 
total number of $f$-weighted
                  compositions with $k$ parts.   
Let $M=\sum_{s\ge 0}f(s)$. Then
\begin{align*}
  {k\brack 0}_{1,f} \equiv M\pmod{2}
\end{align*}
for all $k>0$. 
\end{theorem}
\begin{proof}
  Consider the equation $(\sum_{s\in\nn}f(s)x^s)^k=\sum_{n\ge 0}\binom{k}{n}_{f}x^n$ over $\z/p\z$. Plug in $x=[1]\in\z/2\z$. 
\end{proof}
\begin{remark}
  Note that the previous theorem generalizes the fact that the number
  of odd entries in row $k$ in Pascal's triangle is a multiple of $2$.
\end{remark}
\begin{example}
  In the triangle in Remark \ref{rem:triangle}, note that
  $M=\sum_{s\ge 0} f(s)=5+0+2+1=8$, so that every row sum in the
  triangle (except the
  first) must be even. 
\end{example}

\section{Applications: Prime criteria}\label{sec:applications}
We conclude with two prime criteria for weighted integer
compositions, or, equivalently, extended binomial
coefficients. Babbage's prime criterion (see Granville \cite{Granville:1997} for
references) for ordinary binomial 
coefficients states that an integer $n$ is prime if and only if
$\binom{n+m}{n}\equiv 1\pmod{n}$ for all integers $m$ satisfying $0\le
m\le n-1$. 
The sufficiency of this criterion critically depends on the fact that
the entries $\binom{p}{r}$ in the $p$-th row in Pascal's triangle
are equal to $0$ or $1$ modulo $p$ and the fact that, for ordinary binomial
coefficients, $f(s)=0$ for all $s>1$. 
Hence, this criterion is not expected to 
hold for arbitrary $f$. Indeed,
if $n$ is prime, then, for example, $\binom{n+1}{n}_f\equiv
f(0)f(1)+f(n)f(0)\pmod{n}$ by Corollary \ref{cor:pplus1},  and then,
by repeated application of the corollary and the Vandermonde
convolution,  
$\binom{n+2}{n}_f \equiv f(0)\left(f(0)f(1)+\sum_{i\ge
  0}f(i)f(n-i)\right)\pmod{n}$, etc.\ ---  
and it seems also not obvious how to 
generalize the criterion. 

Conversely, Mann and Shanks' \cite{Mann:1972}
prime criterion allows a straightforward generalization to
weighted integer compositions. We state the criterion and sketch a
proof. 
\begin{theorem}
  Let $f(0)=f(1)=1$. Then, an integer $n>1$ is prime if and only if
  $m$ divides $\binom{m}{n-2m}_f$ for all integers $m$ with $0\le 2m\le
  n$.
\end{theorem}
\begin{proof}[Proof sketch]
If $n$ is prime, then by Theorem \ref{theorem:divis},
$\binom{m}{n-2m}_f\equiv 0\pmod{\frac{m}{\gcd(m,n-2m)}}$. Since $m<n$
and $n$ is prime, then $\gcd(m,n-2m)=\gcd(m,n)=1$. 

Conversely, 
if $n$ is not prime, then, if $n$ is even,
$\binom{n/2}{0}_f=f(0)^{n/2}=1$ and so $m=n/2$ does not divide
$\binom{m}{n-2m}_f$. If $n$ is odd and composite, let $p$ be a prime
divisor of $n$ and choose $m=(n-p)/2=pr$, for a positive integer
$r$. Thus, $\binom{m}{n-2m}_f = \binom{pr}{p}_f$, and by Theorem
\ref{theorem:ms}, $\binom{pr}{p}_f \equiv \binom{pr}{p}\pmod{pr}$ under
the outlined conditions on $f$. Then, Mann and Shanks show that
$\binom{pr}{p}\not\equiv 0\pmod{pr}$.  
\end{proof}
In an earlier work \cite{Eger:2014}, we have derived all steps of the last theorem via application of 
\eqref{eq:repr}. 
\begin{example}
  Let $f(0)=1$, and $f(s)=s$ for all $s\ge 1$. Then, as a primality
  test, e.g., for the integer
  $n=5$, the theorem demands to consider whether $0\mid\binom{0}{5}_f=0$,
  $1\mid\binom{1}{3}_f=3$, and $2\mid\binom{2}{1}_f=2$
  hold true (clearly, the first two of these tests are
  unnecessary). Similarly, the primality test for $n=6$ would be to consider
  whether 
  $0\mid\binom{0}{6}_f=0$, $1\mid\binom{1}{4}_f=4$, $2\mid\binom{2}{2}_f=5$,
  and $3\mid\binom{3}{0}_f=1$ hold true.

  As Mann and Shanks \cite{Mann:1972} point out, the theorem is mainly
  of theoretical rather than practical interest since to determine
  whether the involved row numbers divide the respective binomial
  coefficients may require similarly many computations as in a
  primality test based on Wilson's
  theorem. Also, for practical purposes, one would always want to
  apply the theorem in the setting of ordinary binomial coefficients
  ($f(s)=0$ for all $s>1$).  
\end{example}

\bibliography{lit}{}
\bibliographystyle{plain}

\medskip

\bigskip
\hrule
\bigskip

\noindent 2010 {\it Mathematics Subject Classification}:
Primary 05A10; Secondary 05A17, 11P83, 11A07.

\noindent \emph{Keywords: } 
integer composition, weighted integer composition, 
colored integer composition, 
divisibility, extended binomial coefficient, congruence.

\bigskip
\hrule
\bigskip

\noindent (Concerned with sequences
\seqnum{A007318}, and
\seqnum{A027907}.)

\end{document}